\documentclass[11pt]{amsart}
\usepackage[T1]{fontenc}
\usepackage[utf8]{inputenc}
\usepackage{lmodern}
\usepackage{microtype}
\usepackage{amsmath,amssymb,amsthm,mathtools}
\usepackage[a4paper,margin=31mm]{geometry}
\usepackage[colorlinks=true,linkcolor=blue,citecolor=blue,urlcolor=blue]{hyperref}
\usepackage{enumitem}

\numberwithin{equation}{section}

\newtheorem{theorem}{Theorem}[section]
\newtheorem{proposition}[theorem]{Proposition}
\newtheorem{lemma}[theorem]{Lemma}
\newtheorem{corollary}[theorem]{Corollary}
\newtheorem{remark}[theorem]{Remark}
\newcommand{\R}{\mathbb{R}}
\newcommand{\Q}{\mathbb{Q}}
\newcommand{\N}{\mathbb{N}}
\newcommand{\Bri}{\mathcal{B}}
\newcommand{\NB}{\mathcal{N}}
\newcommand{\cH}{\mathcal{H}}
\newcommand{\cP}{\mathcal{P}}
\newcommand{\LogCap}{\operatorname{Cap}}

\title[Logarithmic exponents of the non-Brjuno set]
{Exact logarithmic Hausdorff and capacity exponents of the non-Brjuno set}
\author[S.~Bazarbaev]{Sardor Bazarbaev}
\address{National University of Uzbekistan named after Mirzo Ulugbek,
4 University Street, 100174 Tashkent, Uzbekistan}
\email{uzedu.bazarbaev@gmail.com}

\author[K.~Rakhimov]{Karim Rakhimov}
\address{V.~I.~Romanovskiy Institute of Mathematics,
Academy of Sciences of the Republic of Uzbekistan,
9 Universitet Street, 100174 Tashkent, Uzbekistan}
\email{karimjon1705@gmail.com}

\date{}


\begin{document}

\begin{abstract}
Let $\mathcal N$ be the set of non-Brjuno real numbers. We determine the
exact critical exponent of $\mathcal N$ for both logarithmic Hausdorff
measure and logarithmic capacity. For the gauges
$h_\delta(r)=(\log(1/r))^{-\delta}$, the critical exponent is $2$:
the $h_\delta$-Hausdorff measure is infinite locally for $0<\delta\leq2$
and vanishes globally for $\delta>2$, while the logarithmic capacity is
positive exactly for orders $0<s\leq2$. In particular,
$\dim_{\mathrm{cap}}\mathcal N=2$.
\end{abstract}
\maketitle

\section{Introduction}

Let
$$
 f(z)=e^{2\pi i\alpha}z+a_2z^2+a_3z^3+\cdots
$$
be a holomorphic germ at the origin, where $\alpha\in\R\setminus\Q$.
The arithmetic condition introduced by Brjuno gives a sufficient condition
for analytic linearization, and Yoccoz proved its sharpness for the quadratic
family; see \cite{B71,Y95}.  If $P_n/Q_n$ are the continued-fraction
convergents of $\alpha$, then $\alpha$ is a Brjuno number precisely when
\begin{equation}\label{eq:Brjuno-intro}
 \sum_{n=0}^{\infty}\frac{\log Q_{n+1}}{Q_n}<\infty.
\end{equation}
We denote the set of Brjuno irrational numbers by $\Bri$, and write
$ \NB:=\R\setminus\Bri$.

For $\delta>0$, define the dimension function
\begin{equation}\label{eq:h-delta}
 h_\delta(0):=0,
 \qquad
 h_\delta(r):=
 \begin{cases}
  \displaystyle\left(\log\frac{1}{r}\right)^{-\delta},
      &0<r<\dfrac1e,\\[6pt]
  1,  &r\geq\dfrac1e.
 \end{cases}
\end{equation}

For $s>0$ and $z,w\in \mathbb C$ define the following kernel
$$ k_s(z,w):=(h_1(|z-w|))^{-s}.$$
Let $K\subset\mathbb C$ be  compact. Let us denote by $\cP(K)$  the set of Borel probability measures supported on
$K$. For $\mu\in\cP(K)$, its
\textit{$s$-energy} is defined by
$$
I_s(\mu):=\iint k_s(z,w)\,d\mu(z)\,d\mu(w).
$$
The logarithmic capacity of order $s$ of $K$ is then defined by
$$
 \LogCap_s(K)
 :=\left(\inf_{\mu\in\cP(K)}
          I_s(\mu)\right)^{-1}.
$$ For arbitrary $E\subset\mathbb C$ set
$$
 \LogCap_s(E):=
 \sup\{\LogCap_s(K):K\subset E,\ K\text{ compact}\}.
$$
$\LogCap_s(E)$ is called logarithmic capacity of $E$ with respect to the kernel $k_s(z,w)$. Since $k_t\geq k_s$ for $t>s$, we have
$$
\LogCap_t(E)\leq\LogCap_s(E).
$$
The corresponding capacity dimension is
\begin{equation}\label{eq:capacity-dimension-intro}
 \dim_{\mathrm{cap}}E
 :=\sup\{s>0:\LogCap_s(E)>0\}.
\end{equation}

Sadullaev and Rakhimov studied the logarithmic capacity of the non-Brjuno
set and proved the following upper estimate.
\begin{theorem}[Sadullaev--Rakhimov \cite{AS1}]\label{thm:SR}
For every $s>2$, we have $\LogCap_s(\NB)=0$.
Consequently,
$\cH^{h_\delta}(\NB)=0$
{for every} $\delta>2$.
\end{theorem}
The second assertion follows from the first by the standard Frostman
comparison; see, for example, \cite{L72}. A recent iterated-logarithmic refinement was obtained by Akramov and
Ashirov \cite{AA26}. The remaining questions concern the critical exponent $2$ and the corresponding
lower bounds. We answer these questions and determine the exact critical
exponent for both the logarithmic Hausdorff measure and logarithmic capacity
in the following theorem.

\begin{theorem}\label{thm:main}
Let $\NB=\R\setminus\Bri$ be the non-Brjuno set, and let $J\subset\R$ be a
nondegenerate interval.
\begin{enumerate}[label=\textup{(\roman*)}]
 \item For every $0<\delta\leq2$, we have
 $$
  \cH^{h_\delta}(\NB\cap J)=\infty.
 $$

 \item For every $\delta>2$, we have
 $$
  \cH^{h_\delta}(\NB)=0.
 $$
 \item The critical exponent for logarithmic capacity is $2$; more precisely,
$$
\LogCap_s(\NB\cap J)>0 \quad \text{for } 0<s\leq2,
\qquad
\LogCap_s(\NB\cap J)=0 \quad \text{for } s>2.
$$
 Consequently,
 $$
  \LogCap_2(\NB)>0,
  \qquad
  \dim_{\mathrm{cap}}(\NB\cap J)
  =\dim_{\mathrm{cap}}\NB=2.
 $$
\end{enumerate}
\end{theorem}
As noted above, part~\textup{(ii)} and the vanishing
$$
\LogCap_s(\NB\cap J)=0, \qquad s>2,
$$
follow directly from the main result of \cite{AS1}. The remaining assertions
are established by two independent lower-bound arguments. First, we consider
the limsup set
$$\mathcal W
 :=\left\{x\in[0,1]:
 \left|x-\frac pq\right|<\frac{e^{-q}}{q^2}
 \text{ for infinitely many }(p,q)\in\mathbb Z\times\mathbb N
 \right\}.$$
Every irrational point of $\mathcal W$ has infinitely many convergents
whose individual contributions to the Brjuno series are bounded away from
zero; hence $\mathcal W\subset\NB$.  Khinchin's theorem (see \cite{Kh64}) and the Mass
Transference Principle (see \cite[Theorem~2]{BV06}) then give
$\cH^{h_2}(\mathcal W\cap J)=\infty$, and the whole range
$0<\delta\leq2$ follows simply from $h_\delta\geq h_2$.

The critical capacity does not follow from the Hausdorff estimate by the
standard Frostman comparison, because that comparison loses a strict amount
in the exponent.  We therefore construct, inside every interval, a compact
subset of $\mathcal W$ supporting a probability measure with finite
order-two energy.  The main arithmetic ingredient is Proposition~\ref{prop:prime-refinement},
based on a refinement using prime denominators.

The paper is organized as follows. In Section~2, we recall the necessary
background on continued fractions, Hausdorff measures, metric Diophantine
approximation, and the Mass Transference Principle. In Section~3, we prove
the lower bound for the logarithmic Hausdorff measure and determine the
critical Hausdorff exponent. In Section~4, we prove positivity of the
logarithmic capacity at the critical exponent and complete the proof of the
main theorem. The proof of the prime-denominator refinement proposition is
given at the end of Section~4.

\subsection*{Acknowledgements} This research was supported by the Ministry of Higher Education, Science and
Innovation of the Republic of Uzbekistan through the Fundamental Research
Projects IL-5421101746 and FL-9524115114.

\section{Preliminaries}

Let $x\in\R\setminus\Q$ and write
$$
 x=[a_0;a_1,a_2,\ldots].
$$
Its \textit{convergents} are denoted by
$$
 \frac{P_n}{Q_n}=[a_0;a_1,\ldots,a_n].
$$
The convergents $P_n/Q_n$ are always written in reduced form. Rational
approximants $p/q$ occurring below, however, are not necessarily reduced. We use the following standard facts; see, for example,
\cite{Kh64,C57}.

\begin{lemma}\label{lem:cf-facts}
Let $x\in\R\setminus\Q$, and let $P_n/Q_n$ be its convergents.
Then $Q_n\to\infty$,
\begin{equation}\label{eq:cf-two-sided}
 \frac{1}{Q_n(Q_n+Q_{n+1})}
 <\left|x-\frac{P_n}{Q_n}\right|
 <\frac{1}{Q_nQ_{n+1}},
\end{equation}
and every  rational number $a/b$ satisfying
$$ \left|x-\frac ab\right|<\frac{1}{2b^2}
$$
is a convergent of $x$.
\end{lemma}

The last assertion is Legendre's theorem.  An irrational number $x$ is
\textit{Brjuno} when the series in \eqref{eq:Brjuno-intro} converges.  It is easy to check that the sets $\Bri$ and $\NB$ are Borel and moreover
\begin{equation}\label{eq:periodicity}
 x\in\Bri\quad\Longleftrightarrow\quad x+m\in\Bri
 \qquad \text{for }m\in\mathbb Z.
\end{equation}

\subsection{Hausdorff measures and metric approximation}

A \emph{dimension function} is a continuous nondecreasing function
$h:[0,\infty)\to[0,\infty)$ with $h(0)=0$.  Consider a cover of $E$ by a finite or countable collection of open balls  $\{B_j:=B(x_j, r_j)\}_{j=1}^m$ such that $r_j < \varepsilon$ for all $1 \le j \le m$, where $m\in \mathbb N\cup\{\infty\}$ depends on the chosen cover. Define 
$$\mathcal H^{h}(E,\varepsilon)=\inf\left\{\sum_{j=1}^m h(r_j):\text{ } \bigcup_{j=1}^m{B_j}\supset E \right\}.$$
As $\varepsilon\downarrow0$, the quantity
$\mathcal H^h(E,\varepsilon)$ is nondecreasing. Hence the limit
$$
\mathcal H^h(E)
:=\lim_{\varepsilon\to0+}\mathcal H^h(E,\varepsilon)
$$
exists and is called the \textit{$h$-Hausdorff measure} of $E$.  When $h_{\alpha}(t)=t^\alpha,$  $\alpha>0$,  the measure $\mathcal H^{h_\alpha}(E)$ is known as the  classical $\alpha$-Hausdorff measure of $E$.

Each $h_\delta$ in \eqref{eq:h-delta} is a dimension function. Moreover,
$$
\lim_{r\to0+}\frac{h_\delta(r)}{r}
=\infty$$
and $r\mapsto h_\delta(r)/r$ is monotonic for all sufficiently small $r$.

We shall use the divergence part of the following Khinchin's theorem.

\begin{theorem}[Khinchin \cite{Kh64}]\label{thm:Khinchin}
Let $\psi:\N\to[0,\infty)$ be eventually nonincreasing, such that
$\sum\limits_{q=1}^\infty\psi(q)=\infty$.  Then the set
$$
 W(\psi):=
 \left\{x\in[0,1]:
 \left|x-\frac pq\right|<\frac{\psi(q)}q
 \text{ for infinitely many }(p,q)\in\mathbb Z\times\mathbb N
 \right\}
$$
has full Lebesgue measure in $[0,1]$.
\end{theorem}

Let $B=B(x,r)=(x-r,x+r)$ be the ball (interval) of radius $r>0$ centered at $x\in\mathbb R$, and let $h$ be a dimension function.  We write
$B^h:=B(x,h(r))$. For a sequence of sets $(E_i)$, we write
$$
\limsup_{i\to\infty} E_i
:=
\bigcap_{N=1}^{\infty}\bigcup_{i\ge N} E_i.
$$
Thus, a point belongs to $\limsup\limits_{i\to\infty}E_i$ if and only if it
belongs to infinitely many of the sets $E_i$. The following is the one-dimensional case of the theorem of Beresnevich and Velani
\cite[Theorem~2]{BV06}.

\begin{theorem}[Mass Transference Principle \cite{BV06}]\label{thm:MTP}
Let $\Omega\subset\R$ be an interval and let $(B_i)$ be a sequence of
balls with radii tending to zero.  Let $h$ be a dimension function such
that $r\mapsto h(r)/r$ is monotonic for all sufficiently small $r$.
Assume that
$$
 \left|J\cap\limsup_{i\to\infty}B_i^h\right|=|J|
$$
for every interval $J\subset\Omega$.  Then
$$
 \cH^h\left(J\cap\limsup_{i\to\infty}B_i\right)=\cH^h(J)
$$
for every interval $J\subset\Omega$.
\end{theorem}

\section{The critical logarithmic Hausdorff exponent}

Define
\begin{equation}\label{eq:W-def}
 \mathcal W
 :=[0,1]\cap
 \bigcap_{N=1}^{\infty}\bigcup_{q\geq N}\bigcup_{p=0}^{q}
 B\left(\frac pq,\frac{e^{-q}}{q^2}\right).
\end{equation}
Thus $x\in\mathcal W$ when
$\lvert x-p/q\rvert<e^{-q}/q^2$ for infinitely many pairs with
$q\geq1$ and $0\leq p\leq q$.

\begin{lemma}\label{lem:W-subset}
Let $\mathcal W$ be as in \eqref{eq:W-def}. Then we have
$$
 \mathcal W\subset\NB\cap[0,1].
$$
\end{lemma}

\begin{proof}
By definition the rational points belong to $\NB$.  Let therefore
$x\in\mathcal W\setminus\Q$. Then, for infinitely many pairs $(p,q)$, we have
\begin{equation}\label{eq:exp-approximation}
 \left|x-\frac pq\right|<\frac{e^{-q}}{q^2}.
\end{equation}
Reduce $p/q$ to $a/b$.  Since $b\leq q$,
$$
 \left|x-\frac ab\right|
 <\frac{e^{-q}}{q^2}
 \leq\frac{e^{-b}}{b^2}
 <\frac{1}{2b^2}.
$$
The  denominators $b$ are unbounded.  Otherwise only finitely many
 rationals could occur, and one of them would be represented along a
sequence $q\to\infty$; then \eqref{eq:exp-approximation} would force $x$
to equal that rational.  By Legendre's theorem, the resulting reduced
fractions are infinitely many convergents of $x$.

Let $P_n/Q_n$ be one of these convergents.  From
\eqref{eq:cf-two-sided} and the preceding estimate,
$$
 \frac{1}{Q_n(Q_n+Q_{n+1})}
 <\frac{e^{-Q_n}}{Q_n^2},
$$
so
$$
 Q_{n+1}>Q_n(e^{Q_n}-1).
$$
Consequently,
$$
 \frac{\log Q_{n+1}}{Q_n}
 >1+\frac{\log Q_n+\log(1-e^{-Q_n})}{Q_n}.
$$
The right-hand side tends to $1$.  Hence infinitely many terms of the
Brjuno series are larger than $1/2$, and the series diverges.  Thus
$x\in\NB$.
\end{proof}

\begin{theorem}\label{thm:H2-W}
For every nondegenerate interval $J\subset[0,1]$, we have
$$
 \cH^{h_2}(\mathcal W\cap J)=\infty.
$$
\end{theorem}

\begin{proof}
Set
$$
 r_q:=\frac{e^{-q}}{q^2}.
$$
The $h_2$-transforms of the balls in \eqref{eq:W-def} have radii
$$
 h_2(r_q)=\frac{1}{(q+2\log q)^2}
         =\frac{\psi(q)}q,
$$
where
$$\psi(q):=\frac{q}{(q+2\log q)^2}\asymp\frac1q.$$
Thus $\sum\limits_{q=1}^\infty\psi(q)=\infty$.  Moreover, $\psi$ is eventually decreasing,
since
$$
 \frac{d}{dx}\log\frac{x}{(x+2\log x)^2}
 =-\frac1x+o\!\left(\frac1x\right).
$$
Khinchin's theorem therefore implies that
$$
[0,1]\cap
\limsup_{q\to\infty}\,
\bigcup_{p\in\mathbb Z}
B\left(\frac pq,h_2(r_q)\right)
$$
has full Lebesgue measure in $[0,1]$. We now show that, for all
sufficiently large $q$, the numerators $p\in\mathbb Z$ occurring in this
limsup may be restricted to $0\leq p\leq q$.

Indeed, for all sufficiently large
$q$, we have
$h_2(r_q)<\frac1q.$
Let $x\in[0,1]$. If $p\notin(-1, q+1)$, then
$$\left|x-\frac pq\right|\geq\frac1q.$$
Consequently, for sufficiently large $q$, the inequality
$$
\left|x-\frac pq\right|<h_2(r_q)
$$
can hold for $x\in[0,1]$ only when
$0\leq p\leq q$.

Thus, after discarding finitely many values of $q$, which does not affect
the limsup set, we obtain
$$
[0,1]\cap
\limsup_{q\to\infty}\,
\bigcup_{p\in\mathbb Z}
B\left(\frac pq,h_2(r_q)\right)
=
[0,1]\cap
\limsup_{q\to\infty}\,
\bigcup_{p=0}^{q}
B\left(\frac pq,h_2(r_q)\right).
$$
Therefore the limsup of the transformed balls appearing in
\eqref{eq:W-def} has full Lebesgue measure in $[0,1]$.
The Mass Transference Principle therefore gives
$$
\cH^{h_2}(\mathcal W\cap J)=\cH^{h_2}(J)=+\infty,
$$
since $J$ is a nondegenerate interval and $h_2(r)/r\to\infty$ as
$r\to 0+$.
\end{proof}

\begin{corollary}\label{cor:H-lower}
For every nondegenerate interval $J\subset\R$ and every
$0<\delta\leq2$, we have
$$
 \cH^{h_\delta}(\NB\cap J)=\infty.
$$
\end{corollary}

\begin{proof}
By \eqref{eq:h-delta}, one has $h_\delta\geq h_2$ whenever
$0<\delta\leq2$.  Thus Lemma~\ref{lem:W-subset} and Theorem~\ref{thm:H2-W} give the assertion on
every nondegenerate subinterval of $[0,1]$.

For a general interval $J\subset\R$, choose $m\in\mathbb Z$ and a
nondegenerate interval $J_0\subset[0,1]$ such that $m+J_0\subset J$.
By \eqref{eq:periodicity}, $m+\mathcal W\subset\NB$, and Hausdorff measure
is translation invariant.  Applying the result on $J_0$ proves the claim.
\end{proof}

\section{Critical capacity and the proof of the main theorem}

The following technical proposition is the arithmetic core of the endpoint
construction.  
\begin{proposition}\label{prop:prime-refinement}
There exists an absolute constant $C>0$ such that the following holds.
Let $I\subset(0,1)$ be a closed interval with $0<|I|<1/e$, let
$\eta>0$, and let $Q_0\geq2$. Then there exist a finite index set $A$, pairwise disjoint intervals
$$
J_a=\overline{B}\left(\frac{p_a}{q_a},
\frac{e^{-q_a}}{8q_a^2}\right)
\subset\operatorname{int}I,
\qquad a\in A,
$$
where each $q_a\geq Q_0$ is prime, together with weights $w_a>0$
satisfying
$$
\sum_{a\in A}w_a=1,
\qquad
\max_{a\in A}w_a\leq\frac12,
$$
such that
\begin{equation}\label{eq:refinement-diagonal}
\sum_{a\in A}w_a^2
\log^2{|J_a|}
\leq\eta,
\end{equation}
and
\begin{equation}\label{eq:refinement-off-diagonal}
\sum_{\substack{a,b\in A\\a\neq b}}w_aw_b
\sup_{x\in J_a,\,y\in J_b}
\log^2{|x-y|}
\leq
C\log^2{|I|}.
\end{equation}
\end{proposition}
We prove Proposition \ref{prop:prime-refinement} at the end of this section
and first use it to prove the following theorem.
\begin{theorem}\label{thm:critical-capacity-W}
For every nondegenerate interval $J\subset[0,1]$, we have
$$
 \LogCap_2(\mathcal W\cap J)>0.
$$
\end{theorem}

\begin{proof}
Choose a nondegenerate closed interval
$I_\varnothing\subset J\cap(0,1)$ with
$|I_\varnothing|<1/e$.  We construct a rooted tree of nested
closed intervals.  Suppose the intervals $I_v$ at level $n-1$ have been
chosen.  In each $I_v$, apply Proposition~\ref{prop:prime-refinement} with
$\eta=2^{-n}$ and with $Q_0$ larger than every denominator used at the
preceding levels.  The finitely many intervals obtained in this way are used to define the
next level of the construction.
Thus the denominators tend to infinity along every branch.  Let $\mathcal I_n$ denote the finite collection of intervals at level $n$,
with $\mathcal I_0:=\{I_\varnothing\}$.

Assign mass $m_\varnothing=1$ to the root.  For each interval $I_u$ obtained from $I_v$ at the next level, define its
mass by
\begin{equation}\label{eq:tree-mass}
m_u:=m_vw_u,
\end{equation}
where $w_u$ is the weight assigned to $I_u$ by
Proposition~\ref{prop:prime-refinement}.
The consistent cylinder masses define a Borel probability measure $\mu$
on
$$
 K:=\bigcap_{n=0}^{\infty}\bigcup_{|v|=n}I_v.
$$
The set $K$ is compact and nonempty.  Since every conditional weight is at
most $1/2$, the mass of a level-$n$ cylinder is at most $2^{-n}$, so
$\mu$ has no atoms.  Moreover, every $x\in K$ lies, at each level, in an
interval centered at a rational $p_n/q_n$ with $q_n\to\infty$ and
$$
 \left|x-\frac{p_n}{q_n}\right|
 \leq\frac{e^{-q_n}}{8q_n^2}
 <\frac{e^{-q_n}}{q_n^2}.
$$
Hence
\begin{equation}\label{eq:K-in-W}
 K\subset\mathcal W\cap J.
\end{equation}

For $n\geq0$, set
$$S_n:=
\sum_{I_v\in\mathcal I_n}
m_v^2
\log^2{|I_v|}.$$
By \eqref{eq:refinement-diagonal} and \eqref{eq:tree-mass}, for $n\geq1$ we have
\begin{align*}
S_n
&=
\sum_{I_v\in\mathcal I_{n-1}}
m_v^2
\sum_{\substack{I_u\in\mathcal I_n\\ I_u\subset I_v}}
w_u^2
\log^2{|I_u|} \\
&\leq
2^{-n}
\sum_{I_v\in\mathcal I_{n-1}}m_v^2
\leq 2^{-n}.
\end{align*}
Also
$ S_0=\log^2{|I_\varnothing|}<\infty$.

All intervals appearing in the construction have length less than $1/e$.
Moreover, along every nested sequence
$$
I_{v_0}\supset I_{v_1}\supset I_{v_2}\supset\cdots
$$
arising from the construction, the diameters tend to zero. Hence, for any
two distinct points $x,y\in K$, there is a unique interval $I_v$ of the
construction which contains both $x$ and $y$, but such that $x$ and $y$
belong to two distinct intervals at the next level inside $I_v$. We call
$I_v$ their last common interval.

Since $\mu$ is nonatomic, the diagonal
$$
\Delta:=\{(x,x):x\in K\}
$$
has zero $\mu\times\mu$-measure. Therefore, up to a set of
$\mu\times\mu$-measure zero, the set $K\times K$ can be decomposed
according to the last common interval of each pair $(x,y)$.

Fix an interval $I_v$, and let $(I_u)$ be the finitely many intervals at the
next level obtained from $I_v$. Consider the set
$$
E_v:=\{(x,y)\in K\times K:\ I_v
\text{ is the last common interval containing }x\text{ and }y\}.
$$
If $(x,y)\in E_v$, then both $x$ and $y$ belong to $I_v$, but they cannot
belong to the same interval at the next level. Indeed, if there were an
interval $I_u$ at the next level containing both $x$ and $y$, then $I_v$
would not be their last common interval. Hence there exist two distinct
intervals $I_u$ and $I_{u'}$ at the next level such that
$$
x\in K\cap I_u,
\qquad
y\in K\cap I_{u'},
\qquad
u\neq u'.
$$
Therefore,
$$
E_v\subset
\bigcup_{\substack{u,u'\\u\neq u'}}
(K\cap I_u)\times(K\cap I_{u'}).
$$
Since the kernel
$\log^2{|x-y|}$ is nonnegative, we obtain
\begin{align*}
\iint_{E_v}
\log^2{|x-y|}
\,d\mu(x)\,d\mu(y) \leq
\sum_{\substack{u,u'\\u\neq u'}}
\int_{K\cap I_u}\int_{K\cap I_{u'}}
\log^2{|x-y|}
\,d\mu(x)\,d\mu(y).
\end{align*}
Thus, in order to estimate the contribution to the energy from all pairs
whose last common interval is $I_v$, it is enough to sum the interactions
between distinct intervals arising from $I_v$ at the next level.

Since
$$
\mu(K\cap I_u)=m_u=m_vw_u,
$$
 we have
$$
\sum_{\substack{u,u'\\u\neq u'}}
\int_{K\cap I_u}\int_{K\cap I_{u'}}
\log^2{|x-y|}
\,d\mu(x)\,d\mu(y)\le m_v^2
\sum_{\substack{u,u'\\u\neq u'}}
w_uw_{u'}
\sup_{\substack{x\in I_u\\y\in I_{u'}}}
\log^2{|x-y|}.
$$
Applying \eqref{eq:refinement-off-diagonal} inside $I_v$, we obtain
$$
\sum_{\substack{u,u'\\u\neq u'}}
\int_{K\cap I_u}\int_{K\cap I_{u'}}
\log^2{|x-y|}
\,d\mu(x)\,d\mu(y)
\leq
C m_v^2
\log^2|I_v|.
$$
Summing over the tree gives
\begin{align*}
  I_2(\mu)&=\int_{K}\int_{K}
\log^2{|x-y|}
\,d\mu(x)\,d\mu(y)
 \leq 
C\sum_{n=0}^{\infty}
\sum_{|v|=n}
m_v^2
\log^2|I_v|\\
&= C\sum_{n=0}^{\infty}S_n
 \leq C\left(\log^2{|I_\varnothing|}+
             \sum_{n=1}^{\infty}2^{-n}\right)
 <\infty.   
\end{align*}

Thus $\mu$ has finite order-two logarithmic energy. Therefore $K$ has positive order-two capacity, and
\eqref{eq:K-in-W} proves the theorem.
\end{proof}

\begin{corollary}\label{cor:critical-capacity-NB}
For every nondegenerate interval $J\subset\R$,
$$
 \LogCap_2(\NB\cap J)>0.
$$
\end{corollary}

\begin{proof}
Choose $m\in\mathbb Z$ and a nondegenerate interval
$J_0\subset[0,1]$ such that $m+J_0\subset J$.  By
Theorem~\ref{thm:critical-capacity-W}, there is a compact set
$K\subset\mathcal W\cap J_0$ carrying a finite order-two energy measure.
Translation preserves that energy, while
$m+\mathcal W\subset\NB$ by \eqref{eq:periodicity} and
Lemma~\ref{lem:W-subset}.  Hence $m+K\subset\NB\cap J$ has positive capacity.
\end{proof}
We are now ready to complete the proof of the main theorem.
\begin{proof}[Proof of Theorem~\ref{thm:main}]
Part~\textup{(i)} is Corollary~\ref{cor:H-lower}. Part~\textup{(ii)} follows from
Theorem~\ref{thm:SR}.

By Corollary~\ref{cor:critical-capacity-NB}, we have
$\LogCap_2(\NB\cap J)>0$ and hence $\LogCap_2(\NB)>0$. Since $\LogCap_s$ is nonincreasing in $s$, for $0<s<2$ we have
$$
 \LogCap_s(\NB\cap J)>0\qquad(0<s\leq2).
$$
For $s>2$, Theorem~\ref{thm:SR} and monotonicity with respect to the underlying
set give
$$
 \LogCap_s(\NB\cap J)\leq\LogCap_s(\NB)=0.
$$
This shows that the critical exponent for logarithmic capacity is $2$.  The identity
$\dim_{\mathrm{cap}}(\NB\cap J)=2$ follows from
\eqref{eq:capacity-dimension-intro}; the global identity follows in the same
way.  
\end{proof}

\subsection{Proof of Proposition \ref{prop:prime-refinement}}

For a closed interval $I\subset(0,1)$, let $I^*$ denote the concentric
interval of length $|I|/2$.  If $q$ is prime, define
$$\mathcal R_q(I)
 :=\left\{\frac pq\in I^*:1\leq p\leq q-1\right\},
 \qquad
 N_q(I):=\#\mathcal R_q(I).$$
Let us first prove the following lemma.
\begin{lemma}\label{lem:grid-estimate}
There is an absolute constant $C\geq1$ such that the following holds.  Let
$I\subset(0,1)$ be a closed interval of length
$0<\ell<1/e$, and let $q,q'\geq8/\ell$ be primes.  Then
\begin{equation}\label{eq:grid-count}
 C^{-1}\ell q\leq N_q(I)\leq C\ell q.
\end{equation}
If $q\neq q'$, then
\begin{equation}\label{eq:grid-cross}
 \sum_{x\in\mathcal R_q(I)}
 \sum_{y\in\mathcal R_{q'}(I)}
 \log^2{|x-y|}
 \leq Cqq'\ell^2\log^2{\ell},
\end{equation}
and
\begin{equation}\label{eq:grid-same}
 \sum_{\substack{x,y\in\mathcal R_q(I)\\x\neq y}}
 \log^2{|x-y|}
 \leq Cq^2\ell^2\log^2{\ell}.
\end{equation}
\end{lemma}

\begin{proof}
The number of points of $q^{-1}\mathbb Z$ in $I^*$ differs from
$\ell q/2$ by at most two.  Since $q\geq8/\ell$, one has, more
precisely,
$$
 \frac{\ell q}{4}\leq N_q(I)\leq\frac{3\ell q}{4},
$$
which implies \eqref{eq:grid-count}.

We shall use the elementary estimate
\begin{equation}\label{eq:log-sum}
 \sum_{j=1}^{M}\log^2\frac{D}{j}
 \leq C M\log^2\frac{D}{M}
 \qquad \text{for } 1\leq M\leq\frac{D}{e}.
\end{equation}
Indeed, write
$\log(D/j)=\log(D/M)+\log(M/j)$.  Since
$\log(D/M)\geq1$, comparison of the remaining sum with
$M\int_0^1\log^2t\,dt$ gives \eqref{eq:log-sum}.

Assume first that $q\neq q'$.  For
$x=p/q\in\mathcal R_q(I)$ and
$y=p'/q'\in\mathcal R_{q'}(I)$, set
$$
 k=pq'-p'q.
$$
The map $(p,p')\mapsto k$ is injective.  Indeed, equality for two pairs
implies that $q\mid(p_1-p_2)$ and $q'\mid(p'_1-p'_2)$, and the allowed
ranges force both differences to vanish.  Moreover, $k\neq0$ and
$$
 |x-y|=\frac{|k|}{qq'}.
$$
There are
$M=N_q(I)N_{q'}(I)\asymp\ell^2qq'$ distinct nonzero signed integers of
this form.  By the upper bound in \eqref{eq:grid-count} and
$\ell<1/e$, one has $M\leq qq'/e$.  Since the summand decreases with
$|k|$, rearrangement (with an inessential factor two for the two signs)
and \eqref{eq:log-sum}, with $D=qq'$, give
$$
 \sum_{x\in\mathcal R_q(I)}
 \sum_{y\in\mathcal R_{q'}(I)}
\log^2{|x-y|}
 \leq C M\log^2\frac{qq'}M
 \leq Cqq'\ell^2\log^2{\ell}.
$$
This proves \eqref{eq:grid-cross}.

For a fixed denominator $q$, each difference $k=p-p'$ occurs at most
$N_q(I)$ times.  Moreover, \eqref{eq:grid-count} and $\ell<1/e$
give $N_q(I)\leq q/e$.  Hence
$$
 \sum_{\substack{x,y\in\mathcal R_q(I)\\x\neq y}}
 \log^2{|x-y|}
 \leq2N_q(I)\sum_{k=1}^{N_q(I)}
 \log^2\frac{q}{k}.
$$
Now \eqref{eq:grid-count} and \eqref{eq:log-sum} imply
\eqref{eq:grid-same}.
\end{proof}

\begin{proof}[Proof of Proposition~\ref{prop:prime-refinement}]
Write $\ell=|I|$.  Choose a large integer
$$
 Q\geq\max\{Q_0,8/\ell\}
$$
and put $R=\lfloor e^{Q/8}\rfloor$.  Let $\mathcal P(Q,R)$ be the set
of primes in $[Q,R]$.  For every
$q\in\mathcal P(Q,R)$ and every
$\xi=p/q\in\mathcal R_q(I)$, let
$$
 J_\xi:=\overline B\left(\xi,\frac{e^{-q}}{8q^2}\right).
$$
For sufficiently large $Q$, these intervals lie in $\operatorname{int}I$
and are pairwise disjoint.  Indeed, their centers lie in $I^*$, while two
distinct  fractions with denominators at most $R$ are separated by
at least $R^{-2}\geq e^{-Q/4}$; the sum of the two corresponding radii is at most
$e^{-Q}/(4Q^2)$.

Set
\begin{equation}\label{eq:normalization}
 Z:=\sum_{q\in\mathcal P(Q,R)}\frac{N_q(I)}{q^2},
 \qquad
 w_\xi:=\frac{q^{-2}}Z
 \quad(\xi\in\mathcal R_q(I)).
\end{equation}
By \eqref{eq:grid-count},
\begin{equation}\label{eq:Z-comparison}
 Z\asymp\ell
 \sum_{q\in\mathcal P(Q,R)}\frac1q.
\end{equation}
Mertens' theorem for primes gives
$$
 \sum_{\substack{q\leq x\\q\text{ prime}}}\frac1q
 =\log\log x+B+o(1);
$$
see \cite[Chapter~22]{HW08}. Therefore, since
$R=\lfloor e^{Q/8}\rfloor$, we have
$$
 \sum_{q\in\mathcal P(Q,R)}\frac1q
 =\log\log R-\log\log Q+o(1)
 =\log Q-\log\log Q-\log 8+o(1)
 \longrightarrow\infty.
$$
Since $\ell$ is fixed, \eqref{eq:Z-comparison} implies that
$Z\to\infty$. Moreover, we have
$$
 \sum_{\xi}w_\xi=1,
 \qquad
 \max_{\xi}w_\xi\leq\frac{1}{Q^2Z}\longrightarrow0.
$$
Thus, after increasing $Q$ if necessary, we have
$\max_\xi w_\xi\leq1/2$. Since $|J_\xi|=e^{-q}/(4q^2)$, one has
$$
 \log\frac{1}{|J_\xi|}=q+2\log q+\log4\leq Cq.
$$
Therefore
$$
 \sum_\xi w_\xi^2\log^2{|J_\xi|}
 \leq\frac{C}{Z^2}
 \sum_{q\in\mathcal P(Q,R)}\frac{N_q(I)}{q^2}
 =\frac{C}{Z}.
$$
Taking $Q$ larger if necessary makes this quantity at most $\eta$,
which proves \eqref{eq:refinement-diagonal}.

For distinct centers $\xi,\zeta$, the preceding separation estimates give
$$ |x-y|\geq\frac12|\xi-\zeta|
 \qquad\text{for }x\in J_\xi,\ y\in J_\zeta,$$
and hence
$$
 \sup_{x\in J_\xi,\,y\in J_\zeta}
\log^2{|x-y|}
 \leq C\log^2{|\xi-\zeta|}.$$
Let
$$
 S:=\sum_{q\in\mathcal P(Q,R)}\frac1q,
 \qquad
 \mathcal A:=\{(q,\xi):q\in\mathcal P(Q,R),\ \xi\in\mathcal R_q(I)\}.
$$
Applying Lemma~\ref{lem:grid-estimate} to equal and unequal denominators gives
\begin{align*}
 &\sum_{\substack{(q,\xi),(q',\zeta)\in\mathcal A\\
                     (q,\xi)\neq(q',\zeta)}}
 q^{-2}(q')^{-2}
\log^2{|\xi-\zeta|}\\
 &\qquad\leq C\ell^2\log^2{\ell}
 \left(S^2+\sum_{q\in\mathcal P(Q,R)}\frac1{q^2}\right)
 \leq 2CS^2\ell^2\log^2{\ell}.
\end{align*}
This is the off-diagonal sum before division by the normalizing factor
$Z^2$.  Since $Z^2\asymp\ell^2S^2$, division proves
\eqref{eq:refinement-off-diagonal} and completes the proof.
\end{proof}

\begin{remark}
The same method applies to the P\'erez--Marco set. Recall that an irrational
number $x$ belongs to the P\'erez--Marco set $\mathcal{PM}$ if
$$
\sum_{n=0}^{\infty}
\frac{\log\log Q_{n+1}}{Q_n}<\infty,
$$
where $P_n/Q_n$ are the convergents of $x$.

Indeed, replacing the approximating radii
$\frac{e^{-q}}{q^2}$ by $\frac{e^{-e^q}}{q^2}$, one obtains a limsup set contained in
$\mathbb R\setminus\mathcal{PM}$. The Mass Transference Principle can then
be applied to the gauges
$$
\widetilde h_\delta(r)
=
\left(\log\log\frac1r\right)^{-\delta},
$$
for sufficiently small $r$, and gives the same critical exponent $2$.
Likewise, the prime-denominator construction used above can be adapted to
the kernel
$$
\widetilde k_s(x,y)
=
\left(\log\log\frac1{|x-y|}\right)^s.
$$
The corresponding upper capacity estimate for the complement of the
P\'erez--Marco set was proved in \cite{R16}. Hence the same critical
exponent $2$ is obtained in this setting.
\end{remark}

\begin{remark}
A higher-dimensional analogue of the upper estimates was recently obtained
by Akramov and Rakhimov in \cite{AR26}. Let $\mathcal B_n$
denote the Brjuno set in $\mathbb C^n$, $n\geq2$. They proved that
$\mathbb C^n\setminus\mathcal B_n$ has zero $C_\sigma$-capacity, with
respect to the kernel
$$
k_\sigma(z,\xi)
=
\frac{\left|\log\|z-\xi\|\right|^\sigma}
     {\|z-\xi\|^{2n-2}},
$$
for every $\sigma>n$. As a consequence,
$$
\mathcal H^{h_\delta}
\bigl(\mathbb C^n\setminus\mathcal B_n\bigr)=0,
\qquad
h_\delta(t)
=
t^{2n-2}\left|\log t\right|^{-\delta},
$$
for every $\delta>n+1$.

It would be interesting to determine whether these estimates are sharp,
in particular at the borderline exponents $\sigma=n$ and $\delta=n+1$.
The one-dimensional arguments used in the present paper rely essentially
on continued fractions and therefore do not directly extend to higher
dimensions.
\end{remark}


\begin{thebibliography}{99}
\bibitem{AA26}
N.~Akramov and A.~Ashirov,
\emph{On the capacity dimensions of the Brjuno and Perez--Marco sets},
J. Anal. \textbf{34} (2026), no.~2, 1071--1082.
\href{https://doi.org/10.1007/s41478-025-00988-5}
{doi:10.1007/s41478-025-00988-5}.
\bibitem{AR26}
N.~Akramov and K.~Rakhimov,
\emph{Capacity dimension of the Brjuno set in $\mathbb C^n$},
Complex Anal. Oper. Theory \textbf{20} (2026), Article~120.
\href{https://doi.org/10.1007/s11785-026-01970-0}
{doi:10.1007/s11785-026-01970-0}.
\bibitem{BV06}
V.~Beresnevich and S.~Velani,
\emph{A mass transference principle and the Duffin--Schaeffer conjecture
for Hausdorff measures},
Ann. of Math. (2) \textbf{164} (2006), no.~3, 971--992.

\bibitem{B71}
A.~D. Brjuno,
\emph{Analytical form of differential equations. I, II},
Trans. Moscow Math. Soc. \textbf{25} (1971), 131--288;
\textbf{26} (1972), 199--239.


\bibitem{C57}
J.~W.~S. Cassels,
\emph{An Introduction to Diophantine Approximation},
Cambridge Tracts in Mathematics and Mathematical Physics, no.~45,
Cambridge University Press, Cambridge, 1957.

\bibitem{HW08}
G.~H. Hardy and E.~M. Wright,
\emph{An Introduction to the Theory of Numbers},
6th ed., Oxford University Press, Oxford, 2008.

\bibitem{Kh64}
A.~Ya. Khinchin,
\emph{Continued Fractions},
University of Chicago Press, Chicago, IL, 1964.


\bibitem{L72}
N.~S. Landkof,
\emph{Foundations of Modern Potential Theory},
Grundlehren der mathematischen Wissenschaften, vol.~180,
Springer-Verlag, Berlin--New York, 1972.
\bibitem{R16}
K.~Rakhimov,
\emph{Capacity dimension of the Perez-Marco set},
in \emph{Topics in Several Complex Variables},
Contemp. Math. \textbf{662},
Amer. Math. Soc., Providence, RI, 2016, 131--138.

\bibitem{AS1}
A.~Sadullaev and K.~Rakhimov,
\emph{Capacity dimension of the Brjuno set},
Indiana Univ. Math. J. \textbf{64} (2015), no.~6, 1829--1834.
\href{https://doi.org/10.1512/iumj.2015.64.5690}
{doi:10.1512/iumj.2015.64.5690}.

\bibitem{Y95}
J.-C. Yoccoz,
\emph{Th\'eor\`eme de Siegel, nombres de Bruno et polyn\^omes quadratiques},
in \emph{Petits diviseurs en dimension $1$},
Ast\'erisque \textbf{231} (1995), 3--88.

\end{thebibliography}
\end{document}